\documentclass[12pt]{article}

\usepackage[english]{babel}
\usepackage{amsmath}
\usepackage{amsthm}
\usepackage{amssymb}
\usepackage{mathrsfs}

\newtheorem{theorem}{Theorem}[section]

\newtheorem{lemma}[theorem]{Lemma}
\newtheorem{proposition}{Proposition}

\title{Optimal stability estimate in the inverse boundary value problem for periodic potentials with partial data}
\date{}
\author{Sombuddha Bhattacharyya\thanks{HKUST Jockey Club Institute for Advanced Study, The Hong Kong University of Science and Technology,
Clear Water Bay, Kowloon, Hong Kong; email: arkatifr@gmail.com} \and C\u{a}t\u{a}lin I. C\^{a}rstea\thanks{HKUST Jockey Club Institute for Advanced Study, The Hong Kong University of Science and Technology,
Clear Water Bay, Kowloon, Hong Kong; email: catalin.carstea@gmail.com}}

\newcommand{\p}[1]{{#1}^{\prime}}

\newcommand{\R}{\mathbb{R}}
\newcommand{\Z}{\mathbb{Z}}
\newcommand{\mH}{\mathcal{H}}

\newcommand{\mT}{\mathcal{T}}

\newcommand{\mU}{\mathcal{U}}
\newcommand{\D}{\text{ d}}
\newcommand{\re}{\mathfrak{Re}}

\begin{document}
\maketitle

\begin{abstract}
We consider the inverse boundary value problem for operators of the form $-\triangle+q$ in an infinite domain $\Omega=\R\times\omega\subset\R^{1+n}$, $n\geq3$, with a periodic potential $q$. For Dirichlet-to-Neumann data localized on a portion of the boundary of the form $\Gamma_1=\R\times\gamma_1$, with $\gamma_1$ being  the complement either of a flat or spherical portion of $\partial\omega$, we prove that a log-type stability estimate holds.
\end{abstract}

\section{Introduction}

For an equation of the type
\begin{equation}
-\triangle u(x)+q(x)u(x)=0,\quad x\in\Omega,
\end{equation}
the inverse boundary value problem is the question of determining the potential $q$, given knowledge of pairs $(u|_{\partial\Omega}, \partial_\nu u|_{\partial\Omega})$ of Dirichlet and Neumann data, either on the whole boundary, or on some proper subset of it.  One way to encode the given information is as the Dirichlet-to-Neumann map $\Lambda_q: u|_{\partial\Omega}\to \partial_\nu u|_{\partial\Omega}$. An interesting sub-problem is the one of uniqueness, i.e. showing that if $\Lambda_{q_1}=\Lambda_{q_2}$, then $q_1=q_2$. A more general question is that of stability: showing that a suitable norm $||q_1-q_2||$ of the difference of two potentials can be controled by a suitable operator norm $||\Lambda_{q_1}-\Lambda_{q_2}||_\ast$ of the difference of the corresponding Dirichlet-to-Neumann maps, through an estimate of the form
\begin{equation}
||q_1-q_2||\leq \phi (||\Lambda_{q_1}-\Lambda_{q_2}||_\ast),\quad \text{where }\lim_{s\searrow 0}\phi(s)=0.
\end{equation}

When full boundary data is given, log-type ($\phi(s)=|\log(s)|^{-\sigma}$) stability estimates have been obtained (see \cite{A}). In \cite{Ma} it has been shown that log-type stability is optimal. For partial boundary data,  log-log-type ($\phi(s) =\left|\log|\log(s)|\right|^{-\sigma}$)  estimates have been obtained (see \cite{Car}, \cite{CDR_2}, \cite{CDR_1},  \cite{CS}, \cite{HW2006}, \cite{Lai}, \cite{Tz}), as well as log-type estimates (see \cite{AK}, \cite{bJ}, \cite{Car}, \cite{Fa}, \cite{HW}). 

The result of Heck and Wang \cite{HW}, where they consider the case of a bounded domain in three or more dimensions and boundary data on a portion of the boundary whose complement is either flat or spherical. In that instance they obtain a log-type stability result. This setup was used by Isakov in \cite{Isakov2007} to prove a uniqueness result. In \cite{Car} a similar method is used to prove a log-type stability result with partial data in the case of electromagnetism. In  \cite{AK}, \cite{bJ}, \cite{Fa} different methods are used, but with the assumption that the unknown coefficients are known near the boundary. In this paper we will follow the method in \cite{HW} to prove a log-type stability result.

Suppose $\omega\subset\R^n$, $n\geq3$, is a bounded domain  with $C^2$-boundary. The domain for the problem we will consider here is an infinite cylinder of the form  $\Omega = \R \times \omega$.  We will denote $\gamma=\partial\omega$ and $\Gamma=\partial\Omega=\R\times\gamma$.  

We consider two types of geometry for $\omega$:
\begin{enumerate}
	\item[(a)] $\omega \subset \{ x_n < 0 \}$ is such that $\gamma_0 = \partial\omega \cap \{  x_n = 0 \}\neq\emptyset$,
	\item[(b)] $\omega \subset B(a,R) = \{x\in\R^n:\lvert x-a \rvert < R\}$ is such that $\gamma_0 = \partial\omega \cap \partial B(a,R)\neq\emptyset$, $\gamma_0 \neq \partial B(a,R)$.
\end{enumerate}
In each of these cases let $\Gamma_0=\R\times\gamma_0$, $\gamma_1=\gamma\setminus\gamma_0$, $\Gamma_1=\R\times\gamma_1=\Gamma\setminus\Gamma_0$.

Let $q \in L^{\infty}(\Omega)$ be real valued and such that
\begin{equation}\label{periodic}
q(x_0+1,\p{x}) = q(x_0,\p{x}), \quad \forall x_0\in\R, x'\in\omega.
\end{equation} 
We consider the following boundary value problem
\begin{equation}\label{BVP}
\left\{\begin{array}{l}
%\mL_q u=
(-\triangle + q)u = 0	\quad 	\mbox{in }\Omega,\\
			 u|_{\Gamma} = f.
\end{array}\right.
\end{equation}
The Dirichlet-to-Neumann map $\Lambda_q$ assigns to the Dirichlet data $f$ the corresponding Neumann data $\Lambda_q(f)=\partial_\nu u|_{\Gamma}$. If we only consider data supported in the open subset $\Gamma_1\subset\Gamma$ of the boundary, then we can define the local Dirichlet-to-Neumann map
\begin{equation}
\Lambda_{q,\Gamma_1}(f) = \partial_{\nu}u\mid_{\Gamma_1}, \quad \text{where }\mbox{supp}(f) \subset \Gamma_1.
\end{equation}

Infinite cylinder domains of this type have been considered in \cite{BKS},   \cite{CKS-t}, \cite{ChS}, \cite{KKS}, \cite{Ki}, \cite{KPS_2}, \cite{KPS_1}, in both  static and time dependent cases.
In \cite{CKS_1}, \cite{CKS_2}, a log-log-type stability result for the potential problem has been obtained.

We will exploit the fact that the potential is periodic in the $x_1$ variable and convert the boundary value problem \eqref{BVP} into a problem on a bounded domain. Then we will establish a relation between the Dirichlet-to-Neumann map for the two problems and then we will prove a stability estimate for the converted problem in the bounded domain and hence we will prove the main result of this article.

\subsection{Main results}

Let
\begin{equation}
C_{\omega} = \sup \{ c>0 : \lVert \nabla u \rVert_{L^2(\omega)} \geq c\lVert u \rVert_{L^2(\omega)}; \forall u \in H^1_{0}(\omega) \},
\end{equation}
 and pick constants $0 < M_{-} < C_{\omega}$, $M_{+} \geq M_{-}$. We will consider potentials in the class
\begin{multline}
\mathcal{V}(M_\pm) = \{ q \in L^{\infty}(\Omega) :\\ q\text{ satisfies\eqref{periodic}, } \lVert q \rVert_{L^{\infty}(\Omega)} \leq M_{+}, \lVert \max(0,-q)\rVert_{L^{\infty}(\Omega)} \leq M_{-} \}.
\end{multline}

We will make use of spaces of the type $\mH^{r,s}(\R \times Y) = H^{r}(\R; H^{s}(Y))$,  where $r,s\geq0$, and $Y \subset \R^n$ could stand for $\omega$, $\gamma$, $\gamma_1$, etc.  Similarly let $\mH^{r,s}_{0}(\R\times Y) = H^{r}(\R ; H^{s}_{0}(Y))$. By $\mH^{-r,-s}(\R\times Y)$ we will denote the dual of $\mH^{r,s}_{0}(\R\times Y)$. We also define the space 
\begin{equation}
H_{\triangle}(\Omega) =  \{ u\in L^2(\Omega): \triangle u \in L^2(\Omega) \},
\end{equation}
with the norm 
\begin{equation}
\lVert u \rVert_{H_{\triangle}(\Omega)}^2 = \lVert u \rVert_{L^2(\Omega)}^2 + \lVert \triangle u \rVert_{L^2(\Omega)}^2.
\end{equation}

For functions $\phi\in C^\infty_0(\bar\Omega)$  we can define the trace operators $\mT_0(\phi) = \phi|_{\Gamma}$ and $\mT_1(\phi) = \partial_\nu\phi|_{\Gamma}$. These extend (see \cite{CKS_1}, Lemma 2.2) to bounded linear operators $\mT_0:H_\triangle(\Omega)\to\mH^{-2,-\frac{1}{2}}(\Gamma)$ and $\mT_1:H_\triangle(\Omega)\to\mH^{-2,-\frac{3}{2}}(\Gamma)$. 

Since $H_\triangle(\Omega)$ is a larger space than $H^2(\Omega)$, it is not entirely straightforward to identify the range of these trace maps in terms of classic function spaces. We will define $\mathscr{H}(\Gamma) = \mT_0 H_\triangle(\Omega)$ as a set. Noticing that $\mT_0$ becomes a bijection onto $\mathscr{H}(\Gamma)$ when restricted to $\mathcal{D} = \{ u\in H_{\triangle}(\Omega) : \triangle u = 0 \}$ (see \cite{CKS_1}, Lemma 2.3), we  endow $\mathscr{H}(\Gamma)$ with the topology $\mathcal{D}$ induces on it through $\mT_0$.

Before stating our stability theorem we need to clarify the well-posedness of the direct problem and give a precise definition of the Dirichlet-to-Neumann map. We have that

\begin{proposition} \label{direct-problem}
	Given fixed $M_+$, $M_-$ and $q \in \mathcal{V}(M_\pm)$
	\begin{enumerate}
		\item[(a)]	For any $f \in \mathscr{H}(\Gamma)$, there is unique $u \in H_{\triangle}(\Omega)$ solving \eqref{BVP} and $C>0$ depending on $\omega$ and $M_{+}, M_{-}$ such that  
				\begin{equation}\begin{aligned}
					\lVert u \rVert_{L^2(\Omega)} 
					\leq C\lVert f\rVert_{\mathscr{H}(\Gamma)}.	
				\end{aligned}\end{equation} 
		\item[(b)]	The Dirichlet to Neumann map $\Lambda_{q} : f \to \mT_{1}$ is a bounded operator from $\mathscr{H}(\Gamma)$ into $\mH^{-2,-\frac{3}{2}}(\Gamma)$.
		\item[(c)]	For any $\tilde{q} \in \mathcal{V}$, the operator $\Lambda_{q} - \Lambda_{\tilde{q}}$ is bounded from $\mathscr{H}(\Gamma)$ to $L^2(\Gamma)$.
	\end{enumerate}
\end{proposition}
 
 This is identical to the statement of \cite[Proposition 1.1]{CKS_1}. Though there $\omega\subset\R^2$, their proof does not rely crucially on the dimension and does apply equally well to our $\omega\subset\R^n$ case. 
 %For the sake of completeness, we include a proof of Proposition \ref{direct-problem} in an appendix.

Let $||\cdot||_\ast=||\cdot||_{\mathscr{H}(\Gamma)\to L^2(\Gamma)}$.   Our main result is
\begin{theorem}\label{main-theorem}
	Let $\Omega = \R \times \omega \subset \R^{1+n}$, where $\omega$ satisfies one of the geometry constraints  $(a)$ or $(b)$. Let $M_+$, $M_-$, $N$ be fixed. If $q_1, q_2 \in \mathcal{V}\cap H^s((0,1)\times\omega)$ for an $s>\frac{1+n}{2}$, and $||q_1||_{H^s((0,1)\times\omega)},||q_2||_{H^s((0,1)\times\omega)}<N$,
	then there exists $C>0$ and $\sigma>0$  such that 
	\begin{equation}\label{stability-estimate}
	\lVert q_1 - q_2 \rVert_{L^\infty(\Omega)} 
	\leq C\left\lvert \log \lVert\Lambda_{q_1,\Gamma_1} - \Lambda_{q_2,\Gamma_1}\rVert_{*} \right\rvert^{-\sigma}.
	\end{equation} 
\end{theorem}

We will prove  Theorem \ref{main-theorem}by making use of the Floquet-Bloch-Gel’fand (FBG) transform (or fiber transform). This will allow us to prove Theorem \ref{main-theorem} by proving an equivalent result for a bounded domain. We describe this in section \ref{fiber}. 
In section \ref{cgo-solutions} we introduce complex geometric optics solutions for case (a), which we then, in section \ref{estimate}, use to establish our stability estimate. Finally, in section \ref{kelvin} we make use of a (partial) Kelvin transform to reduce case (b) to case (a). 

\section{Fiber decomposition}\label{fiber}

In this section we summarize certain results concerning the FBG transform. All statements are easy generalizations  of results proven in \cite{CKS_1}. 

Let $Y$ be  $\omega$, $\partial\omega$, or $\gamma_1$. We define the operator
\begin{align}\label{definition of U}
\mU(f)_{\theta}(x_0,\p{x}) =  \sum_{k \in \Z} e^{ik\theta}f(x_0+k,\p{x}), \quad f\in C^{\infty}(\R \times Y), \quad \theta \in [0,2\pi).
\end{align}
This extends to a unitary operator mapping $L^2(\R \times Y)$ onto the direct sum $\int_{(0,2\pi)}^{\oplus}L^2((0,1)\times Y) \frac{\D \theta}{2\pi}$. We will use the notation $\check Y=(0,1)\times Y$.

We need to introduce several function spaces. 
%For $s>\frac{1}{2}$, let
%\begin{equation}
%\mH^s_{\theta}(\check Y) =  \{ u \in H^s(\check Y):
%\partial^{j}_{x_0}u(1,\cdot) - e^{i\theta}\partial^{j}_{x_0}u(0,\cdot) = 0, j<s-1/2 \}
%\end{equation}
%and for $0\leq s\leq\frac{1}{2}$ let 
%$
%\mH^s_{\theta}(\check Y) =  H^s(\check Y). 
%$
Let 
\begin{equation}
H_{\triangle,\theta}(\check\omega) =  \{ u \in H_{\triangle}(\check\omega): \partial^{j}_{x_1}u(1,\cdot) - e^{i\theta}\partial^{j}_{x_1}u(0,\cdot) = 0,\ j= 0,1 \}. 
\end{equation}
Also let $\mH^{s,t}_\theta(\check Y)$ be the set of all functions
\begin{equation}
\phi(x)=\sum_{k\in\mathbb{Z}}e^{i(\theta+2\pi k)x_0}\phi_k(x'),\quad \phi_k\in H^t( Y),\quad \sum_{k\in\mathbb{Z}}(1+k^2)^s||\phi_k||^2_{H^{t}(Y)}\leq\infty.
\end{equation}

The maps $u\to u|_{\check\gamma}$ and $u\to \partial_\nu u|_{\check\gamma}$ defined on smooth functions may be extended to bounded operators
\begin{equation}\begin{aligned}
\mT_{0,\theta}: H_{\Delta,\theta}(\check\omega) \mapsto \mH^{-2,-\frac{1}{2}}(\check\gamma) \quad \mbox{and} \quad
\mT_{1,\theta}: H_{\Delta,\theta}(\check\omega) \mapsto \mH^{-2,-\frac{3}{2}}(\check\gamma).
\end{aligned}\end{equation}
Consider the set
\begin{equation}\begin{aligned}
\mathscr{H}_{\theta}(\check\gamma) =  \{ \mT_{0,\theta}u: u \in H_{\triangle,\theta}(\check\omega) \}.
\end{aligned}\end{equation}
It can be shown that $\mT_{0,\theta}$ is a bijection between $\mathcal{D}_{\theta} =  \{ u\in H_{\triangle,\theta}(\check\omega):\triangle u = 0) \}$ and $\mathscr{H}_{\theta}(\check\gamma)$. As with the original problem, we use this bijection to endow $\mathscr{H}_\theta(\check\gamma)$ with a topology.

Note that if $X_\theta(\check Y)$ is any of the spaces defined above, and $X(\R\times Y)$ is the similarly defined space on $\R\times Y$, then it holds that
\begin{equation}
\mU X(\R\times Y) = \int_{(0,2\pi)}^{\oplus}X_\theta(\check Y) \frac{\D \theta}{2\pi}.
\end{equation}
It also holds that, for $q\in\mathcal{V}(M_{\pm})$, 
\begin{equation}
\mU (-\triangle+q)|_{H_\Delta(\Omega)}\mU^{-1}= \int_{(0,2\pi)}^{\oplus}(-\triangle+q)|_{H_{\Delta,\theta}(\check\omega)} \frac{\D \theta}{2\pi},
\end{equation}
\begin{equation}
\mU \mT_{j}\mU^{-1}= \int_{(0,2\pi)}^{\oplus}\mT_{j,\theta} \frac{\D \theta}{2\pi},\quad j=0,1.
\end{equation}

For any $\theta \in \left[ 0 , 2\pi\right)$ consider the following  boundary value problem in $\check\omega$ 
\begin{equation}\label{BVP-FBG}
\left\{\begin{array}{l}
 (-\triangle+q)	u = 0 \quad \mbox{in } \check\omega,\\
					u = f \quad \mbox{on } \check\gamma,\\
u(1,\cdot) - e^{i\theta}u(0,\cdot) = 0 \quad \mbox{in } \omega,\\
\partial_\nu u(1,\cdot) - \partial_\nu e^{i\theta}u(0,\cdot) = 0 \quad \mbox{in } \omega.
\end{array}\right.
\end{equation}

The following proposition is analogous to  Proposition \ref{direct-problem}.

\begin{proposition}[see\cite{CKS_1}, Propositon 3.2]\label{technical proposition-2}
	Let $\theta \in [0,2\pi)$ and fix $M_+$, $M_-$ and $q \in \mathcal{V}(M_\pm)$. Then
	\begin{enumerate}
		\item[(a)]	For any $f \in \mathscr{H}_{\theta}(\check\gamma)$, there exists unique $u \in H_{\triangle,\theta}(\check\omega)$ solving \eqref{BVP-FBG} with
		\begin{equation}\begin{aligned}
				\lVert u \rVert_{L^2(\check\omega)} \leq C\lVert f\rVert_{\mathscr{H}_{\theta}(\check\gamma)}.
		\end{aligned}\end{equation}
		\item[(b)]	The DN map $\Lambda_{q,\theta}: f \mapsto \mT_{1,\theta}u\mid_{\check\gamma_1}$ is a bounded operator from $\mathscr{H}_{\theta}(\check\gamma)$ to $\mH^{-2,-\frac{3}{2}}(\check\gamma_1)$.
		\item[(c)]	For $q,\tilde{q} \in \mathcal{V}(M_\pm)$, the operator $\Lambda_{q,\theta} - \Lambda_{\tilde{q}, \theta}$ is bounded from $\mathscr{H}_{\theta}(\check\gamma)$ to $L^2(\check\gamma)$.	
	\end{enumerate}
\end{proposition}

We have that
\begin{equation}\begin{aligned}
	\mU \Lambda_{q,\Gamma_1} \mU^{-1} = \int_{(0,2\pi)}^{\oplus} \Lambda_{q,\check\gamma_1,\theta} \frac{\D \theta}{2 \pi},
\end{aligned}\end{equation} 
and
\begin{equation}\label{FBG staility relation}
\lVert \Lambda_{q_1,\Gamma_1} - \Lambda_{q_2,\Gamma_1}\lVert_{*} = \sup_{\theta \in [0,2\pi)} \lVert \Lambda_{q_1,\check\gamma_1,\theta} - \Lambda_{q_2,\check\gamma_1,\theta} \rVert_{*},
\end{equation}
where on the right hand side $||\cdot||_\ast$ denotes the operator norm $||\cdot||_{\mathscr{H}_{\theta}\to L^2}$

To prove Theorem \ref{main-theorem} it is then enough to prove
\begin{theorem}\label{main-theorem-FBG}
	Under the same conditions as in Theorem \ref{main-theorem}, we have
	\begin{equation}\begin{aligned}
		\lVert q_1 - q_2 \rVert_{L^\infty} \leq C_{\theta} \left\lvert \log\lVert \Lambda_{q_1,\Gamma_1,\theta} - \Lambda_{q_2,\Gamma_1,\theta}\rVert_{*}\right\rvert^{-\sigma}, 
	\end{aligned}\end{equation}
	holds for $\theta \in [0,2\pi)$.
\end{theorem}

\section{Complex geometric optics solutions}\label{cgo-solutions}

In this section we will construct complex geometric optics solutions for the problem  \eqref{BVP-FBG}. These will later be used to prove Theorem \ref{main-theorem-FBG} and hence Theorem \ref{main-theorem}.

In this section we consider the case $(a)$, i.e $\omega \subset \{ x_n < 0 \}$ is such that $\gamma_0  =  \partial\omega \cap \{  x_n = 0 \}\neq\emptyset$. If $x=(x_0, x_1, \ldots, x_n)\in \R^{1+n}$, we will  use the notation $x^\ast=(x_0, x_1, \ldots, -x_n)$. For a function $f$ defined on $\R^{1+n}$ or a proper subset of it, we will write $f^\ast(x) = f(x^\ast)$. 

We extend $q_j$ so that
as $q_j=0$ for $ x \in \{\R^{1+n} : x_n < 0 \} \setminus \check\omega$ and $q_j(x) = q_j(x^{*})$ for   $x\in \{\R^{1+n} : x_n > 0 \}$. This means that $q^\ast_j=q_j$.

For each $j=1,2$ we are interested in solutions $u_j \in H_{\triangle}(\check\omega)$  of \eqref{BVP-FBG}  for $q=q_j$  of the form
\begin{equation}\label{CGO}
u_j(x_0,\p{x}) = e^{\zeta_j\cdot x}(1+r_j),   
\end{equation}
where $\zeta_j \in \mathbb{C}^{1+n}$ is chosen so that $\zeta_j\cdot\zeta_j = 0$.

We will  construct $\zeta_j$ explicitly.
Let $\xi,\eta \in \R^n\setminus\{0\}$ such that $\lvert \xi \rvert = 1$, $\xi\cdot \eta = 0$ and $\xi\cdot(0,\ldots,0,1) = 0$. For any $0 \leq \theta < 2\pi$, $k \in \mathbb{Z}+\frac{1}{2}$, $r>0$ and $\xi$, $\eta$ as above we define
\begin{equation}\label{definition-L}
l=\left(\theta + 2\pi\left([r]+\sigma_k\right)\right)\left(1, -\frac{2\pi k}{|\eta|^2}\eta \right),
\end{equation}
where $[r]$ denotes the integer part of $[r]$, and $\sigma_k=7/4$ if $k-1/2$ is even and $\sigma_k=5/4$ if $k-1/2$ is odd. 
Let
\begin{equation}\label{definition-tau}
\tau  =  \sqrt{\frac{\lvert \eta \rvert^2}{4} + \pi^2k^2 + \lvert l \rvert^2}.
\end{equation}
%Then
%\begin{equation}\label{estimate-tau}
%2\pi r < \tau \leq \frac{\lvert (2\pi k,\eta)\rvert}{2} + 4\pi(r+1)\left( \frac{\lvert 2\pi k\rvert}{\lvert \eta \rvert} + 1 \right).
%\end{equation}
We define
\begin{equation}\begin{aligned}
\zeta_1 &= \left( i\pi k, -\tau\xi + i\frac{\eta}{2} \right) + il,\\
\zeta_2 &= \left(-i\pi k, \tau\xi -i\frac{\eta}{2} \right) + il,
\end{aligned}\end{equation}
and observe that 
\begin{equation}\label{zeta-properties}
\begin{gathered}
		\zeta_1 + \overline{\zeta_2} = i(2\pi k,\eta),	\quad\zeta^{*}_1 + \overline{\zeta^{*}_2} = i(2\pi k,\eta^{*}),\\
		\zeta^{*}_1 + \overline{\zeta_2} = i(2\pi k, \frac{1}{2}(\eta+\eta^\ast)+\left(\theta + 2\pi\left([r]+\sigma_k\right)\right)
		\frac{2\pi k}{|\eta|^2}(\eta-\eta^\ast)),\\
		\zeta_1 + \overline{\zeta^{*}_2} = i(2\pi k, \frac{1}{2}(\eta+\eta^\ast)-\left(\theta + 2\pi\left([r]+\sigma_k\right)\right)
		\frac{2\pi k}{|\eta|^2}(\eta-\eta^\ast)),\\
		\zeta_j\cdot\zeta_j = 0 = \zeta^{*}_j\cdot\zeta^{*}_j.
\end{gathered}
\end{equation}

Suppose $R>0$ is such that $\omega\subset \overline{B(0,R)}\subset\R^n$. We follow \cite{Ha, CKS_2} to prove
\begin{lemma}
Solutions of the form \eqref{CGO} exist in $[0,1]\times[-R,R]^n$ such that
\begin{equation}
||r_j||_{L^2([0,1]\times[-R,R]^n)}\leq \frac{C}{\tau}||q_j||_{L^\infty},
\end{equation}
\begin{equation}
||\nabla r_j||_{L^2([0,1]\times[-R,R]^n)}\leq C||q_j||_{L^\infty},
\end{equation}
for $\tau>2\pi R||q_j||_{L^\infty}$.
\end{lemma}
\begin{proof}
  Note that  $u=e^{\zeta\cdot x}(1+r)$, where $\zeta=\zeta_1$ or $\zeta=\zeta_2$ satisfies \eqref{BVP-FBG} if 
\begin{equation}\label{eq-r}
-\triangle r -2\zeta\cdot\nabla r+qr=-q,\; r(1,\cdot)=r(0,\cdot),\; \partial_\nu r(1,\cdot)=\partial_\nu r(0,\cdot).
\end{equation}

We will construct a solution operator $G_\zeta$ for the operator $-\triangle-2\zeta\cdot\nabla$. Without loss of generality we may choose a basis in $\R^n$ so that $\xi=(1,0\ldots,0)\in\R^n$. For any $\alpha=(\alpha_0,\alpha')\in\left(\mathbb{Z}^{1+n}-(0,1/2,0,\ldots,0)\right)$, let
\begin{equation}
e_\alpha=(2R)^{-n/2}\exp\left\{2\pi i\alpha_0x_0+\frac{i\pi}{R}\alpha'\cdot x'    \right\}.
\end{equation}
These form an orthonormal basis in $L^2([0,1]\times[-R,R]^n)$. Suppose we want to find a solution $r$ to the equation
$$
-\triangle r-2\zeta\cdot\nabla r=\phi.$$
If we define 
$
\hat r_\alpha=\langle r,e_\alpha\rangle
$
and $\hat\phi_\alpha=\langle \phi,e_\alpha\rangle $ then
\begin{equation}
\hat r_\alpha=\frac{\hat \phi_\alpha}{\pi^2\left(4\alpha_0^2+R^{-2}{\alpha'\cdot\alpha'}-
4i\pi^{-1}\zeta_0\alpha_0-2i(\pi R)^{-1}\zeta'\cdot\alpha'\right)}.
\end{equation}
In our case we have $\re(\zeta_0)=0$, $\re(\zeta')=-\tau(1,0,\ldots,0)\in\R^n$, so we get
\begin{equation}
|\hat r_\alpha|\leq\frac{|\hat\phi_\alpha|}{\frac{2\tau}{\pi R}|\alpha_1|}\leq\frac{\pi R}{\tau}|\hat\phi_\alpha|
\end{equation}
for all $\alpha$. Writing $r=G_\zeta \phi$ and using Plancherel's equality we get
\begin{equation}
||G_\zeta \phi||_{L^2([0,1]\times[-R,R]^n)}\leq\frac{\pi R}{\tau}||\phi||_{L^2([0,1]\times[-R,R]^n)}.
\end{equation}
We may also obtain (see \cite{Ha})  estimates of the form
\begin{equation}
||\nabla G_\zeta\phi||_{L^2([0,1]\times[-R,R]^n)}\leq C_1||\phi||_{L^2([0,1]\times[-R,R]^n)},
\end{equation}
\begin{equation}
||\triangle G_\zeta\phi||_{L^2([0,1]\times[-R,R]^n)}\leq C_2(\zeta)||\phi||_{L^2([0,1]\times[-R,R]^n)},
\end{equation}
where $C_1$ doesn't depend on $\zeta$.

We may write \eqref{eq-r} in the form
\begin{equation}
r+G_\zeta (q r)=- G_\zeta (q).
\end{equation}
For $\tau>2\pi R||q||_{L^\infty}$ we may invert the operator $I+G_\zeta(q\cdot)$ that appears on the left hand side to obtain a solution $r\in H^2([0,1]\times[-R,R]^n)$ that satisfies
\begin{equation}
||r||_{L^2([0,1]\times[-R,R]^n)}\leq \frac{CR}{\tau}||q||_{L^\infty},
\end{equation}
with a constant $C$ that depends only on $\omega$. We also get
\begin{equation}
||\nabla r||_{L^2([0,1]\times[-R,R]^n)}\leq C||q||_{L^\infty},
\end{equation}
\end{proof}

Let
 \begin{equation}
v_j=u_j-u_j^\ast= e^{\zeta_j\cdot x}(1+r_j)-e^{\zeta_j^\ast\cdot x}(1+r_j^\ast).
\end{equation}
Clearly, $\text{supp }v_j|_{\check\gamma}\subset\check\gamma_1$.

%Also observe that using the Carleman estimate in \cite[Proposition: 5.1]{CKS_1} for $u \in \mC^2_{\theta}(\overline{\check\omega})$, with $u \mid_{\check\gamma} = 0 = \partial_{\nu}u \mid_{\check\gamma}$, we get
%\begin{equation}\label{C-Estimate}
%\tau^2 \lVert e^{-\tau\xi\cdot \p{x}} u \rVert^2_{L^2(\check\omega)} 
%\leq \lVert e^{-\tau\xi\cdot \p{x}} \triangle u \rVert^2_{L^2(\check\omega)}.
%\end{equation}

Using integration by parts  we have, for $q = (q_1-q_2)$,
\begin{multline}%\label{Integral-Id}
\langle (\Lambda_{q_1,\check{\gamma}_1,\theta} - \Lambda_{q_2,\check{\gamma}_1,\theta}) v_1, v_2 \rangle\\
= \int_{\check\omega} qv_1 \overline{v_2} \D x 
= \int_{\check\omega} q
\left( e^{\left(\zeta_1 + \overline{\zeta_2}\right)\cdot x} 
+ e^{\left(\zeta^{*}_1 + \overline{\zeta^{*}_2}\right) \cdot x} 
\right) \D x\\
-\int_{\check\omega} q
\left(
e^{\left(\zeta^{*}_1 + \overline{\zeta_2}\right)\cdot x} 
+ e^{\left(\zeta_1 + \overline{\zeta^{*}_2}\right)\cdot x} \right) \D x \\
+ \int_{\check\omega} q\left( e^{\left(\zeta_1 + \overline{\zeta_2}\right)\cdot x}(r_1 + \overline{r_2} + r_1\overline{r_2})
+ e^{\left(\zeta^{*}_1 + \overline{\zeta^{*}_2}\right)\cdot x}(r^{*}_1 + \overline{r^{*}_2} + r^{*}_1\overline{r^{*}_2}) \right) \D x \\
- \int_{\check\omega} q\left( e^{\left(\zeta^{*}_1 + \overline{\zeta_2}\right)\cdot x}(r^{*}_1 + \overline{r_2} + r^{*}_1\overline{r_2}) 
+ e^{\left(\zeta_1 + \overline{\zeta^{*}_2}\right)\cdot x}(r_1 + \overline{r^{*}_2} + r_1\overline{r^{*}_2}) \right) \D x\\
=: I_1-I_2+I_3-I_4.
\end{multline}

 We can easily estimate
\begin{multline}
\left\lvert I_3 \right\rvert 
\leq \left\lvert \int_{\check\omega} q e^{i(2\pi k,\eta)\cdot x}\left( r_1 + \overline{r_2} + r_1\overline{r_2} \right) \D x \right\rvert \\
+ \left\lvert \int_{\check\omega} q e^{i(2\pi k,\eta^{*})\cdot x}\left( r^{*}_1 + \overline{r^{*}_2} + r^{*}_1\overline{r^{*}_2} \right) \D x \right\rvert
\leq \frac{C}{\tau}\lVert q \rVert_{L^{\infty}(\check\omega)}.
\end{multline}
Since $\re(\zeta_1^\ast+\overline{\zeta_2})=0$, the we get in the same way that
\begin{equation}
\left\lvert I_4 \right\rvert 
\leq \frac{C}{\tau}\lVert q \rVert_{L^{\infty}(\check\omega)}.
\end{equation}

Using the fact that $q^\ast=q$, we see that 
\begin{equation}\label{Principal term of Int-Id}
I_1 = 2\int_{\check\omega} e^{i(2\pi k,\eta)\cdot x} q(x) \D x
= 2\int_{\R^n} e^{i\eta\cdot \p{x}} \left(\int_{0}^{1} e^{2\pi i k x_1} q(x_1,\p{x}) \D x_1 \right) \D \p{x} 
= 2\widehat{q_k}(\eta), 
\end{equation}
where
\begin{equation}
q_k(\p{x}) = \int_{0}^{1} e^{2\pi i k x_1}q_(x_1,\p{x}) \D x_1.
\end{equation}

Similarly
\begin{equation}
I_2=2\hat q_k(\kappa),
\end{equation}
where $\kappa\in\R^n$ is 
\begin{equation}
\kappa=\frac{1}{2}(\eta+\eta^\ast)+(\theta+2\pi([r]+\sigma_k))\frac{2\pi k}{|\eta|^2}(\eta-\eta^\ast),
\end{equation}
where $\sigma_k$ is either $5/4$ or $7/4$. 
%Since $(\eta-\eta^\ast)\cdot(\eta+\eta^\ast)=0$ we can derive the estimate
%\begin{equation}
%|\kappa|\geq \min(|\eta|, 4\pi(\theta+2\pi([r]+\sigma))|k| ).
%\end{equation} 

\begin{lemma}
There exists  constants $C,\epsilon_0,\alpha>0$ such that for any $0<\epsilon<\epsilon_0$
\begin{equation}
|\hat q_k(\rho)|\leq C [\exp[-\frac{\epsilon^2}{4\pi}(k^2+|\rho|^2)]+\epsilon^\alpha]
\end{equation}
\end{lemma}
\begin{proof}
%We will apply .
Since $q\in H^s(\check\omega)$, $s>(1+n)/2$ it follows that there is an $\alpha>0$ such that $q\in C^{0,\alpha}(\overline{\check\omega})$. We will denote by $\tilde q$ the extension by zero of $q$ to $\R^{1+n}$. First we estimate, for $|y_0|<1$, 
\begin{equation}
||\tilde q(\cdot+y_0,\cdot)-\tilde q(\cdot)||_{L^1(\R^{1+n})}=\left(\int_{1-|y_0|}^1 +\int_0^{|y_0|}\right)\int_\omega|q|\leq 2||q||_{L^\infty(\check\omega)}|\omega|\;|y_0|.
\end{equation}
Applying \cite[Lemma 2.2]{HW} (see also \cite[Lemma 2.4]{CH})we also obtain that there exists $C,\delta>0$ such that if $y'\in\R^n$, $|y'|<\delta$ then
\begin{equation}
||\tilde q(\cdot,\cdot+y')-\tilde q(\cdot)||_{L^1(\R^{1+n})}\leq C|y|^\alpha.
\end{equation}
Using the triangle inequality we can conclude that for $y=(y_0,y')$
\begin{equation}
||q(\cdot+y)-q(\cdot)||_{L^1(\R^{1+n})}\leq C|y|^\alpha.
\end{equation}
We can then apply \cite[Lemma 2.1]{HW} to obtain the conclusion.
\end{proof}

A consequence of this is that 
\begin{equation}
|I_2|\leq C \left[\exp\left[-\frac{\epsilon^2}{4\pi}\left(k^2+\kappa^2\right)\right]+\epsilon^\alpha\right].
\end{equation}

On the other hand
\begin{multline}
\left\lvert \langle (\Lambda_{q_1,\check{\gamma}_1,\theta} - \Lambda_{q_2,\check{\gamma}_1,\theta}) v_1, v_2 \rangle \right\rvert\leq ||\Lambda_{q_1,\check{\gamma}_1,\theta} - \Lambda_{q_2,\check{\gamma}_1,\theta}||_\ast ||v_1||_{\mathscr{H}(\check\gamma)}||v_2||_{L^2(\check\gamma)}\\
\leq C  ||\Lambda_{q_1,\check{\gamma}_1,\theta} - \Lambda_{q_2,\check{\gamma}_1,\theta}||_\ast ||v_1||_{L^2(\check\omega)}||v_2||_{L^2(\check\omega)}\\
\leq e^{2|\xi|\tau}||\Lambda_{q_1,\check{\gamma}_1,\theta} - \Lambda_{q_2,\check{\gamma}_1,\theta}||_\ast.
\end{multline}
We have here used the fact that 
$$||v_1||_{\mathscr{H}_\theta(\check\gamma)}\leq C||v_1||_{H_{\triangle,\theta}(\check\omega)}\leq C(||v_1||_{L^2(\check\omega)}+||q_1v_1||_{L^2(\check\omega)})\leq C||v_1||_{L^2(\check\omega)}.$$

Putting together the above estimates and the fact that
\begin{equation}
|\kappa|\geq \frac{4\pi^2k r}{|\eta|^2}|\eta-\eta^\ast|, \quad \tau>2\pi r,
\end{equation}
we obtain
\begin{equation}\label{apriori-est}
|\hat q_k(\eta)|\leq C \left[  e^{2\tau}||\Lambda_{q_1,\check{\gamma}_1,\theta} - \Lambda_{q_2,\check{\gamma}_1,\theta}||_\ast +
   \exp\left[-\frac{2\epsilon^2 r^2k^2}{|\eta|^4} |\eta-\eta^\ast|^2 \right]+\epsilon^\alpha
+\frac{1}{r}  \right],
\end{equation}
where $C$ is a constant depending on $n$, $\check\omega$, $M_\pm$.

\section{Stability estimate}

\subsection{Case (a), $\gamma_0\subset\{x_n=0\}$}\label{estimate}

We need the following  lemma:
\begin{lemma}[see \cite{CKS_1}, Lemma 6.3]\label{Technical-lemma}
	Let $q \in L^2((0,1)\times\R^n)$. 
	Then there exists $C>0$ such that 
	\begin{equation}\begin{aligned}
		\lVert q\rVert_{H^{-1}((0,1)\times \R^n)} 
		\leq C\left\lVert \sum_{k\in \mathbb{Z}} \left(1+\lvert (k,\cdot)\rvert^2\right)^{-\frac{1}{2}} \widehat{q_k}(\cdot) \right\rVert_{L^2(\R^n)}.
	\end{aligned}\end{equation}
\end{lemma} 

%\begin{proof}
%	The proof is similar to the proof of \cite[Lemma-6.3]{CKS_1}, however, for the sake of completeness we present the proof here.
%	
%	Let $C>0$ satisfying 
%	\begin{equation}\begin{aligned}
%		\left\lVert \sum_{k \in \mathbb{Z}} (1 + \lvert (k,\cdot) \rvert^2)^{\frac{1}{2}} \hat{w}_k \right\rVert_{L^2(\R^n)} 
%		\leq C \lVert w \rVert_{H^1((0,1) \times \R^n)},
%	\end{aligned}\end{equation}
%	for all $w \in H^1_0((0,1) \times \R^n)$. For $v \in L^2((0,1)\times \R^n)$ we have 
%	\begin{equation}\begin{aligned}
%		\langle v, w \rangle_{H^{-1}((0,1)\times \R^n), H^{1}_{0}((0,1)\times \R^n)}
%		= \langle v, w \rangle_{L^2((0,1)\times \R^n)} 
%		= \int_{\R^n} \sum_{k\in \mathbb{Z}} \hat{v}_k(\eta) \overline{\hat{w}_k(\eta)} \D \eta.
%	\end{aligned}\end{equation}
%	So we get 
%	\begin{equation}\begin{aligned}
%		\left\lvert \langle v, w \rangle_{H^{-1}((0,1)\times \R^n), H^{1}_{0}((0,1)\times \R^n)} \right\rvert
%		\leq C\left\lVert \sum_{k \in\mathbb{Z}} (1+\lvert (k,\cdot) \rvert^2)^{\frac{1}{2}}\hat{v}_k \right\rVert_{L^2(\R^n)} \lVert w \rVert_{H^1((0,1)\times \R^n)}.
%	\end{aligned}\end{equation}
%\end{proof}

Then
\begin{equation}\begin{aligned}
\lVert q \rVert^2_{H^{-1}((0,1)\times \R^n)} 
\leq& C\int_{\R^{1+n}} \left(1+\lvert (k,\eta)\rvert^2\right)^{-1} \lvert\widehat{q_k}(\eta)\rvert^2 \D \mu(k) \D \eta,
\end{aligned}\end{equation}
where $\mu(k) = \sum_{n\in\mathbb{Z}} \delta_{n}$.
Let $$B_{\rho} = \{ (k,\eta) \in \R^{1+n} : |k|<\rho, |\eta+\eta^\ast|<\rho,|\eta-\eta^\ast|<\rho\}.$$
Then we have
\begin{equation}\label{critical integrations}
\begin{aligned}
\lVert q \rVert^2_{H^{-1}((0,1)\times \R^n)} 
\leq& C\int_{\R^{1+n} \setminus B_{\rho}} \left(1+\lvert (k,\eta)\rvert^2\right)^{-1} \lvert\widehat{q_k}(\eta)\rvert^2 \D \mu(k) \D \eta\\
&+ C\int_{B_{\rho}} \left(1+\lvert (k,\eta)\rvert^2\right)^{-1} \lvert\widehat{q_k}(\eta)\rvert^2 \D \mu(k) \D \eta.
\end{aligned}
\end{equation}

The first integral in \eqref{critical integrations} is easy to estimate:
\begin{multline}
\int_{\R^{1+n} \setminus B_{\rho}} \left(1+\lvert (k,\eta)\rvert^2\right)^{-1} \lvert\widehat{q_k}(\eta)\rvert^2 \D \mu(k) \D \eta\\
\leq \frac{C}{\rho^2} \int_{\R^{1+n} \setminus B_{\rho}} \lvert\widehat{q_k}(\eta)\rvert^2 \D \mu(k) \D \eta
\leq \frac{C}{\rho^2}\lVert q \rVert^2_{L^{\infty}(\Omega)}.
\end{multline}
To estimate the second integral, we use \eqref{apriori-est} and the fact that on $B_\rho$
\begin{equation}
\tau\leq 100 (r+\rho)
\end{equation}
 to obtain:
\begin{multline}
\int_{B_{\rho}} \left(1+\lvert (k,\eta)\rvert^2\right)^{-1} \lvert\widehat{q_k}(\eta)\rvert^2 \D \mu(k) \D \eta\\
\leq C [  \rho^{1+n}e^{200(r+\rho))} ||\Lambda_{q_1,\check{\gamma}_1,\theta} - \Lambda_{q_2,\check{\gamma}_1,\theta}||_\ast
+\epsilon^{2\alpha}\rho^{1+n}\\
+\rho^{1+n}r^{-1}+\rho^n \int _{-\rho}^\rho \exp\left[-\frac{\epsilon^2 r^2}{\rho^4} s^2\right] \D s]\\
\leq C [  \rho^{1+n}e^{200(r+\rho)} ||\Lambda_{q_1,\check{\gamma}_1,\theta} - \Lambda_{q_2,\check{\gamma}_1,\theta}||_\ast
+\epsilon^{2\alpha}\rho^{1+n}+\rho^{1+n}r^{-1}+\rho^{n+2}\epsilon^{-1} r^{-1}].
\end{multline}

We can choose $\epsilon$ such that $\epsilon^{2\alpha}=r^{-1}$. Then 
\begin{equation}
\epsilon^{2\alpha}\rho^{1+n}+\rho^{1+n}r^{-1}+\rho^{n+2}\epsilon^{-1} r^{-1}\leq
C[\rho^{1+n}r^{-1}+\rho^{n+2}r^{-1+\frac{1}{2\alpha}}]\leq C \rho^{n+2}r^{-\tilde\alpha},
\end{equation}
with $\tilde\alpha=1-1/(2\alpha)$. Next we choose $r=\rho^{\frac{4+n}{\tilde\alpha}}$. With this choice, going back to \eqref{critical integrations} we obtain
\begin{equation}
\lVert q \rVert^2_{H^{-1}((0,1)\times \R^n)}\leq C\left[           
\rho^{1+n}\exp[400\rho^{\frac{4+n}{\tilde\alpha}}] ||\Lambda_{q_1,\check{\gamma}_1,\theta} - \Lambda_{q_2,\check{\gamma}_1,\theta}||_\ast +\rho^{-2}
\right]
\end{equation}
To finish the estimate, we can choose $\rho$ so that
\begin{equation}
\rho^{1+n}\exp[400\rho^{\frac{4+n}{\tilde\alpha}}] ||\Lambda_{q_1,\check{\gamma}_1,\theta} - \Lambda_{q_2,\check{\gamma}_1,\theta}||_\ast =\rho^{-2}.
\end{equation}
In this case clearly there exists a $\gamma(n,\tilde\alpha)>0$ so
\begin{equation}
\rho\geq \left|\log ||\Lambda_{q_1,\check{\gamma}_1,\theta} - \Lambda_{q_2,\check{\gamma}_1,\theta}||_\ast\right|^{\frac{\gamma}{2}}.
\end{equation}
This gives us the estimate
\begin{equation}
\lVert q \rVert^2_{H^{-1}((0,1)\times \R^n)}\leq C \left|\log ||\Lambda_{q_1,\check{\gamma}_1,\theta} - \Lambda_{q_2,\check{\gamma}_1,\theta}||_\ast\right|^{-\gamma}.
\end{equation}

Since we are additionally assuming that $||q_j||_{H^s(\check\omega)}<N$, for and $s$ that can be written as $s=\frac{1+n}{2}+2\epsilon$, by interpolation we have that there exists $\tau\in(0,1)$ such that
\begin{equation}
||q_1-q_2||_{L^\infty(\check\omega)}\leq ||q_1-q_2||_{H^{\frac{1+n}{2}+\epsilon}(\check\omega)}\leq
 ||q_1-q_2||^\tau_{H^{-1}(\check\omega)}||q_1-q_2||^{1-\tau}_{H^s(\check\omega)}.
\end{equation}
Our desired result now follows trivially.

\subsection{Case (b), $\gamma_0\subset\{|x'-a|=R\}$}\label{kelvin}

Without loss of generality $a=(0,\ldots,0,R)\in\R^n$ and $0\not\in\overline{\omega}$. Following \cite{Isakov2007}, \cite{HW}, we employ the (partial) Kelvin transform 
\begin{equation}
y'=\left(\frac{2R}{|x'|}\right)^2x', \quad y_0=x_0,
\end{equation}
whose inverse is
\begin{equation}
x'=\left(\frac{2R}{|y'|}\right)^2y', \quad x_0=y_0.
\end{equation}
Let $\tilde\omega$, $\tilde\gamma$, $\tilde\gamma_0$, $\tilde\gamma_1$ be the images of $\omega$, $\gamma$, $\gamma_0$, $\gamma_1$ through this transform. Then $\tilde\gamma_0\subset\{y_n=2R\}$, $\tilde\gamma_1=\tilde\gamma\cap\{y_n>2R\}$, so the transformed domain $\tilde\omega$ satisfies the conditions of case (a).

For a function $u(x)$ we define
\begin{equation}
\tilde u(y)=\left(\frac{2R}{|y'|}\right)^{n-2}u(x(y)).
\end{equation}
Note that
\begin{equation}
\left(\frac{|y'|}{2R}\right)^{n+2}\triangle_y\tilde u(y)=\triangle_x u(x).
\end{equation}
If $-\triangle u+qu=0$, then
\begin{equation}
-\triangle\tilde u+\tilde q\tilde u=0,\quad \tilde q(y)=\left(\frac{2R}{|y'|}\right)^4 q(x(y)).
\end{equation}

Let $f\in \mathscr{H}_\theta(\check\gamma_1)$, then there exists $u\in H_{\Delta,\theta}(\check\omega)$ such that $u|_{\check\gamma_1}=f$, $\triangle u=0$. We notice that $\tilde u|_{\check{\tilde{\gamma_1}}}=\tilde f$ and $\triangle\tilde u=0$. 
Since $\left({2R}/{|y'|}\right)^{n-2}$ is a bounded positive function on $\check{\tilde{\omega}}$, there are constants $C',C''>0$ such that
\begin{equation}
C'||u||_{L^2(\check{{\omega}})}\leq ||\tilde u||_{L^2(\check{\tilde{\omega}})}\leq C''||u||_{L^2(\check{{\omega}})}.
\end{equation}
So 
\begin{equation}
C'||f||_{\mathscr{H}_\theta(\check{{\gamma_1}})}\leq ||\tilde f||_{\mathscr{H}_\theta(\check{\tilde{\gamma_1}})}\leq C''||f||_{\mathscr{H}_\theta(\check{{\gamma_1}})}.
\end{equation}
Similarly, for any $g\in L^2(\check\gamma_1)$, 
\begin{equation}
C'||g||_{L^2(\check{{\gamma_1}})}\leq ||\tilde g||_{L^2(\check{\tilde{\gamma_1}})}\leq C''||g||_{L^2(\check{{\gamma_1}})}.
\end{equation}
It follows then that the norms $||\Lambda_{q_1,\check\gamma_1,\theta}-\Lambda_{q_2,\check\gamma_1,\theta}||_\ast$ and
$||\Lambda_{q_1,\check{\tilde{\gamma}}_1,\theta}-\Lambda_{q_2,\check{\tilde{\gamma}}_1,\theta}||_\ast$ are equivalent, i.e. that there are constants $C',C''>0$ such that
\begin{equation}
C'||\Lambda_{q_1,\check\gamma_1,\theta}-\Lambda_{q_2,\check\gamma_1,\theta}||_\ast
\leq       ||\Lambda_{q_1,\check{\tilde{\gamma}}_1,\theta}-\Lambda_{q_2,\check{\tilde{\gamma}}_1,\theta}||_\ast
\leq C'' ||\Lambda_{q_1,\check\gamma_1,\theta}-\Lambda_{q_2,\check\gamma_1,\theta}||_\ast.
\end{equation}
With this observation, we see that the stability estimate we have proved for case (a) implies the one for case (b).

\paragraph{Acknowledgement} Work on this paper began in March 2017, when the first author visited the National Taiwan University National Center for Theoretical Sciences (NCTS), where the second author was employed at the time. We wish to acknowledge the support that NCTS provided towards making this work possible.

\bibliographystyle{plain}
\bibliography{bib-stability-v2}
\end{document}